\documentclass{article}
\usepackage{amssymb}
\usepackage{amsmath}
\usepackage{amsthm}

\numberwithin{equation}{section}

\theoremstyle{plain}
	\newtheorem{thm}{Theorem}[section]
	\newtheorem{lem}{Lemma}[section]

\theoremstyle{definition}

\theoremstyle{remark}
	\newtheorem{rem}{Remark}[section]

\newcommand{\MP}[4]{
\Phi^{#1}_{#2}\left(
\arraycolsep=1pt
\renewcommand{\arraystretch}{0.8}
\begin{array}{cccccccc}#3\end{array}\Big|
\begin{array}{cccccccc}#4\end{array}
\renewcommand{\arraystretch}{1.0}
\right)
}

\begin{document}
\title{\bf Transformation formulas for bilinear sums of basic hypergeometric series}
\author{Yasushi KAJIHARA}
\date{}
\maketitle

{
\allowdisplaybreaks 

\begin{center}
{\it Dedicated to Professor Richard  Askey for his 80th birthday}
\end{center}

\begin{abstract}
A master formula of transformation formulas for bilinear sums of basic hypergeometric series
is proposed. 
It is  obtained from the author's previous results on
a transformation formula for Milne's multivariate generalization of basic hypergeometric
series of type $A$ with different dimensions and it can be considered as a 
generalization of the Whipple-Sears transformation formula for terminating balanced ${}_4 \phi_3$
series. 
As an application of the master formula, the one variable cases of some transformation formulas
for bilinear sums of basic hypergeometric series are given as examples.
The bilinear transformation formulas seem to be new in the literature,
even in one variable case.  
\end{abstract}

\medskip

\section{Introduction}

Classical orthogonal polynomials
(or orthogonal polynomials of (basic) hypergeometric type), which
notion was introduced in G.E~Andrews and R.~Askey \cite{AA}, 
have been investigated in various aspects, and important applications
to other branches of mathematics and related areas were found.
Askey and Wilson \cite{AW} arranged the classical orthogonal polynomials
for $q=1$ in a scheme which soon became known as the {\it Askey scheme}.
A $q$-Askey scheme was also arranged.
For the proofs of various properties of these polynomials,
transformation and summation formulas of (basic) hypergeometric series
turned out to be useful. 

On the other hand, S.C.~Milne \cite{MilneG} has introduced 
a class of multivariate generalizations of basic hypergeometric series
which are nowadays called as $A_n$ basic hypergeometric series
(or basic hypergeometric series in $SU(n+1)$).
Transformation and summation formulas for $A_n$ hypergeometric series 
including their elliptic generalization,
and applications to other branches in mathematics and related areas 
have been investigated by many authors.

Among these results, the present author \cite{Kaji1} obtained a number of 
transformation formulas which relate (mainly basic) hypergeometric series 
of type $A$ with different dimensions
(see also \cite{KajiS}, \cite{KajiR}, \cite{KajiBDT} and \cite{KajiSymm}).
In this paper, we propose a master formula (\eqref{MF} below) of transformation 
formulas for bilinear sums of basic hypergeometric series. 
We obtain it from our  previous results on the Euler transformation formula \cite{Kaji1}
for  basic hypergeometric series of type $A$ with different dimensions 
and it can be considered as a generalization of the Whipple-Sears transformation formula:
\begin{eqnarray}\label{SearsT1}
{}_{4} \phi_3
\left[
	\begin{matrix}
		a, b, c, q^{-N}\\
 		d, e, f,\\
\end{matrix}
;q; q
\right] 
&=& 
\frac{(e/a, d e / b c)_N}
{(e, d e / a b c)_N}
{}_{4} \phi_3
\left[
	\begin{matrix}
		a, d/b, d/c\\
 		d, d f / b c, d e / b c\\
\end{matrix}
; q; q
\right],
\quad 
(a b c = d e f q^{N-1})
\end{eqnarray}
for terminating balanced ${}_4 \phi_3$ series. 
As an application of the master formula \eqref{MF}, we give one variable cases of 
some transformation formulasfor bilinear sums of basic hypergeometric series as examples.
The bilinear transformation formulas involve high degree of freedom of parameters
and seem to be new in the literature, even in the one variable case.
What is remarkable is that transformations for strictly multivariate basic 
hypergeometric series of type $A$, especially with different dimensions, 
may shed a light to future investigations of basic hypergeometric series and 
related classical orthogonal polynomials, even in one variable case.
It is expected that the bilinear transformations in this paper have applications 
to the moment representations(see Ismail and Stanton \cite{IsSt1}, \cite{IsSt2}) 
and the Poisson kernel (see Rahman \cite{R1}, \cite{R2}) 
for classical orthogonal polynomials. Furthermore, it would be interesting 
if the bilinear transformations can be useful for deeper understanding 
of fundamental properties of classical orthogonal polynomials
(see the lecture notes by Askey \cite{AskeyOPSF}), such as orthogonality, addition formulas 
(see Noumi and Mimachi \cite{NouM}),  linearization of the products, connection coefficients, 
positivity and, in particular, convolution structures (see Koelink and van der Jeugt \cite{KoJe1}).

\section{Notation and necessary tools}

\medskip

\indent
In this section, we give some notations for basic hypergeometric series
and we recall some results of our previous work \cite{Kaji1}.
In this paper we basically follow the notation 
from the book by Gasper and Rahman \cite{GRBook}.  
Let $q$ be a complex number under the condition $0 < |q|< 1$.   
Define the $q$-shifted factorial as 

\begin{equation}
(a)_\infty := (a;q)_\infty =\prod_{n \in \mathbb N}(1 - a q^n), \quad 
(a)_k := (a;q)_k = \frac{(a)_\infty}{(a q^k)_\infty} 
\quad \mbox{for} 
\ k \in \mathbb C. 
\end{equation}
Namely unless otherwise indicated, we omit the basis $q$.
We will write
\begin{equation}
(a_1, a_2, \cdots , a_n)_k := (a_1)_k (a_2)_k \cdots (a_n)_k.
\end{equation}
We denote the basic hypergeometric series ${}_{n+1} \phi_n$ as 
\begin{equation}\label{BHSdef}
{}_{n+1} \phi_n \left[
	\begin{matrix}
		a_0, \{ a_i \}_{n}\\
 			\{ c_i \}_{n}
\end{matrix}
;q; u
\right] 
:=
{}_{n+1} \phi_n \left[
	\begin{matrix}
		a_0, a_1, \ldots  a_{n}\\
 			c_1,  \ldots  c_n
\end{matrix}
;q; u
\right] 
= \sum_{k \in \mathbb N} 
\frac{(a_0, a_1, \ldots , a_n)_k }{(q, c_1, \ldots , c_n)_k} u^k.
\end{equation}

A ${}_{n+1} \phi_{n}$ series is called well-poised if 
$a_0 q = a_1 c_1 = \cdots = a_n c_n$. In addition, if  
$a_1 = q \sqrt{a_0}$ and $a_2 = - q \sqrt{a_0}$ then 
the ${}_{r+1} \phi_r $ is called very well-poised.
Throughout of this paper, we denote the very well-poised basic 
hypergeometric series ${}_{r+1} \phi_r $ as 
${}_{r+1} W_r$ series defined by the following:

\begin{eqnarray}\label{DefVWP}
&&{}_{n+1} \phi_n
	\left[
		\begin{matrix}
			& a_0, & q \sqrt{a_0},&  - q \sqrt{a_0},& a_3,& \ldots ,&  a_n \\
			& & \sqrt{a_0},& - \sqrt{a_0},& a_0 q / a_3,& \ldots, 
			&  a_0 q /a_n 
		\end{matrix};
		q; u
	\right]\nonumber\\
&& 
\quad \quad
=  
\sum_{k \in {\mathbb N}} 
\frac{1 - a_0 q^2 }{1 - a_0}
\frac{(a_0)_k (a_3)_k \cdots (a_n)_k}
{(q)_k  (a_0 q / a_3)_k \cdots (a_0 q /a_n)_k}
u^k	
\nonumber
\\
&&
\quad \quad \quad \quad 
:= 
{}_{n+1} W_n
	\left[
		 a_0; a_3, \ldots ,  a_n; 
		 q; u
	\right].
\end{eqnarray}

Our fundamental tool to derive the master formula \eqref{MF} will be the 
multiple Euler transformation formula for basic hypergeometric series of
type $A$:

\begin{thm} {\bf (Theorem 1.1 of \cite{Kaji1})}
\begin{eqnarray}\label{ETG}
&&
\sum_{\gamma \in {\mathbb N}^n}
u^{|\gamma|} 
\frac{\Delta (x q^{\gamma})}{\Delta (x)}
\prod_{1 \le i, j \le n}
\frac{(a_j x_i / x_j)_{\gamma_i}} {(q x_i / x_j)_{\gamma_i}}
\prod_{1 \le i \le n, 1 \le k \le m}
\frac{(b_k x_i y_k )_{\gamma_i}}{(c x_i y_k )_{\gamma_i}}
\\
&& \quad \quad 
= \
\frac{(A B u/ c^m)_\infty}
{(u)_\infty}
\
\sum_{\delta \in {\mathbb N}^m}
 \left( \frac{A B u}{ c^m } \right)^{|\delta|} 
\frac{\Delta (y q^{\delta})}{\Delta (y)}
\nonumber\\
&&  \quad \quad \quad 
\ \times \
\prod_{1 \le k, l \le m}
\frac{((c / b_l^{}) y_k / y_l)_{\delta_k}}{(q y_k / y_l)_{\delta_k}}
\prod_{1 \le i \le n, 1 \le k \le m}
\frac{((c /a_i^{}) x_i y_k )_{\delta_k}}
{(c x_i y_k)_{\delta_k}} 
\nonumber
\end{eqnarray}
where $A := a_1 a_2 \cdots a_n,  B := b_1 b_2 \cdots b_m$ and 
\begin{equation}
\Delta(x) := \prod_{1 \le i < j \le n}
(x_i - x_j)
\qquad  \mbox{and}
\qquad
\Delta(x q^\gamma) := \prod_{1 \le i < j \le n}
(x_i q^{\gamma_i} - x_j q^{\gamma_j})
\end{equation}
are the Vandermonde determinant for the sets of variables $x=(x_1, \cdots , x_n)$
and $x^\gamma = (x_1 q^{\gamma_1}, \cdots,\allowbreak
x_n q^{\gamma_n})$,  respectively.
\end{thm}

\begin{rem}
In the case $m=n=1$ and $x_1 = y_1 = 1$, \eqref{ETG} reduces to the third Heine 
transformation formula for basic hypergeometric series ${}_2 \phi_1$ (a basic analogue of Euler 
transformation formula for the Gauss hypergeometric series ${}_2 F_1$):
\begin{eqnarray}\label{3rdHeine}
{}_{2} \phi_1
\left[
	\begin{matrix}
		a, b\\
 		c\\
\end{matrix}
;q; u
\right] 
&=& 
\frac{(ab u /c)_\infty}
{(u)_\infty}
{}_{2} \phi_1
\left[
	\begin{matrix}
		c/b, c/a\\
 		c\\
\end{matrix}
; q; ab u / c
\right].
\end{eqnarray}
\end{rem}
Here we present the multiple Euler transformation formula \eqref{ETG} in a slightly 
different expression as \cite{Kaji1}.
We note that \eqref{ETG} is valid for any pair of positive integers $n$ and $m$
and \eqref{ETG} is a transformation formula between $A_n$ ${}_{m+1} \phi_{m}$ series
and $A_m$ ${}_{n+1} \phi_{n}$ series. 
The definitions and terminologies for $A_n$ basic hypergeometric series 
can be found in \cite{KajiBDT}.

\section{Master formula}

\medskip

 In this section, we present the master formula \eqref{MF}.

First, notice that \eqref{ETG} is an identity for formal power series of the 
variable $u$.
Set the homogeneous part in $A_n$ basic hypergeometric series in \eqref{ETG} as $\Phi_N^{n,m}$:
\begin{eqnarray}
&&\MP{n,m}{N}
{ \{a_i \}_n \\ \{ x_i \}_n}
{  \{ b_k y_k \}_m \\ \{c y_k \}_m}
:=
\sum_{\gamma \in {\mathbb N}^n, | \gamma | = N}
\frac{\Delta (x q^{\gamma})}{\Delta (x)}
\\
&& \quad \quad \quad \quad \quad \quad 
\times    
\prod_{1 \le i, j \le n}
\frac{(a_j x_i / x_j)_{\gamma_i}} {(q x_i / x_j)_{\gamma_i}}
\prod_{1 \le i \le n, 1 \le k \le m}
\frac{(b_k x_i y_k / x_n y_m)_{\gamma_i}}{(c x_i y_k / x_n y_m)_{\gamma_i}}.
\nonumber
\end{eqnarray}
The multiple Euler transformation \eqref{ETG} can be expressed in terms 
of $\Phi^{n,m}_N$. Namely, it can be stated as: 

\begin{eqnarray}\label{rETG}
&&
\sum_{K \in \mathbb N}
\MP{n,m}{K}
{ \{a_i \}_{n_1} \\ \{ x_i \}_{n_1}}
{  \{ b_k  y_k \}_{m_1} \\ \{ c y_k \}_{m_1}}
u^K
=
\frac{( A B u /c^m )_\infty}{( u )_\infty}
\\
&& \quad \quad \quad \quad \quad
\times 
\sum_{L \in \mathbb N}
\MP{m,n}{L}
{ \{ c / b_k \}_{m_1} \\ \{ y_k \}_{m_1}}
{  \{ (c / a_i )  x_i \}_{n_1} \\ \{ c x_i \}_{n_1}}
\left(
\frac{ A B u }{c^{m}}
\right)^{L}.
\nonumber
\end{eqnarray}

Now consider the product of the multiple series

\begin{eqnarray}\label{ProdM}
&&
\left( \
\sum_{K \in \mathbb N}
\MP{n_1,m_1}{K}
{ \{a_i \}_{n_1} \\ \{ x_i \}_{n_1}}
{  \{ b_k  y_k \}_{m_1} \\ \{ c y_k \}_{m_1}}
u^K \ 
\right)
\\
&& \quad \quad 
\times
\left(  \
\sum_{L \in \mathbb N}
\MP{n_2,m_2}{L}
{ \{ f / e_p \}_{n_2} \\ \{ z_p \}_{n_2}}
{  \{ ( f / d_s ) w_s \}_{m_2} \\ \{ f w_s \}_{m_2}}
\left(
\frac{f^{n_2} u}{DE}
\right)^{L} \ 
\right),
\nonumber
\end{eqnarray}
under the restriction $A B / c^{m_1} = DE /f^{n_2}$.
By virtue of multiple Euler transformation \eqref{ETG}, we obtain the following 
formal power series identity of variable $u$
under the condition above:
\begin{eqnarray}\label{pMF}
&&
\left( \
\sum_{K \in \mathbb N}
\MP{n_1,m_1}{K}
{ \{a_i \}_{n_1} \\ \{ x_i \}_{n_1}}
{  \{ b_k  y_k \}_{m_1} \\ \{ c y_k \}_{m_1}}
u^K \
\right)
\\
&& \times 
\left( \
\sum_{L \in \mathbb N}
\MP{n_2,m_2}{L}
{ \{ f / e_p \}_{n_2} \\ \{ z_p \}_{n_2}}
{  \{ ( f / d_s ) w_s \}_{m_2} \\ \{ f w_s \}_{m_2}}
\left(
\frac{f^{n_2} u}{D E}
\right)^{L}
\ \right)
\nonumber
\\
&=&
\left( \
\sum_{K \in \mathbb N}
\MP{m_1,n_1}{K}
{ \{ c / b_k \}_{m_1} \\ \{ y_k \}_{m_1}}
{  \{ (c / a_i )  x_i \}_{n_1} \\ \{ c x_i \}_{n_1}}
\left(
\frac{c^{m_1} u}{A B}
\right)^{K}
\ \right)
\nonumber
\\
&&
\times
\left( \
\sum_{L \in \mathbb N}
\MP{m_2,n_2}{L}
{ \{ d_s \}_{m_2} \\ \{ w_s \}_{m_2}}
{  \{ e_p  z_p \}_{n_2} \\ \{ f z_p \}_{n_2}}
u^L
\ \right).
\nonumber
\end{eqnarray}
By taking the coefficient of $u^N$ in the identity above,
we arrive at the following:
\begin{thm}
\begin{eqnarray}\label{MF}
&&
\sum_{K \in \mathbb N}
\MP{n_1,m_1}{K}
{ \{a_i \}_{n_1} \\ \{ x_i \}_{n_1}}
{  \{ b_k  y_k \}_{m_1} \\ \{ c y_k \}_{m_1}}
\MP{n_2,m_2}{N-K}
{ \{ f / e_p \}_{n_2} \\ \{ z_p \}_{n_2}}
{  \{ ( f / d_s ) w_s \}_{m_2} \\ \{ f w_s \}_{m_2}}
\left(
\frac{f^{n_2}}{DE}
\right)^{N-K}
\\
&=&
\sum_{L \in \mathbb N}
\MP{m_1,n_1}{L}
{ \{ c / b_k \}_{m_1} \\ \{ y_k \}_{m_1}}
{  \{ (c / a_i )  x_i \}_{n_1} \\ \{ c x_i \}_{n_1}}
\MP{m_2,n_2}{N-L}
{ \{ d_s \}_{m_2} \\ \{ w_s \}_{m_2}}
{  \{ e_p  z_p \}_{n_2} \\ \{ f z_p \}_{n_2}}
\left(
\frac{c^{m_1}}{AB}
\right)^{L}
\nonumber
\end{eqnarray}
under the condition
\begin{equation}
\label{tbc}
AB /c^{m_1} = DE / f^{n_2}.
\end{equation}
\end{thm}
 Hereafter, we call \eqref{MF} the {\it master formula}
and \eqref{tbc} the {\it totally balancing condition}" .

\section{Bilinear transformation formulas}

\medskip

 In this section, we present three transformation formulas for bilinear sums of 
basic hypergeometric series. 

 One of the most remarkable and fundamental features of multiple hypergeometric series is that 
the homogeneous part of multiple hypergeometric series can be expressed in terms of 
(multiple) very well-poised hypergeometric series. 

\begin{lem}
\begin{eqnarray}\label{Phi-W}
&&
\MP{2,m}{N}
{ \{a_i \}_{2} \\ \{ x_i \}_{2}}
{  \{ b_k y_k \}_m \\ \{c y_k \}_m }
=
\frac{(a_{2})_N}{(q)_N}
\frac
{(a_{2} x_{2} / x_1)_N}
{(x_{2} / x_1)_N}
\prod_{1 \le k \le m}
\frac
{(b_{k} x_{2}  y_k)_N}
{(c x_{2} y_k)_N}
\\
&& \
\times \
{}_{2 m + 6} W_{2 m + 5}
\left[
{q^{-N} / x_{2} };
a_1, { \{b_k y_k \}_{m}, a_{2} /x_{2}, }
{\{ (q^{1-N}/ x_{2})c^{-1} y_k^{-1} \}_{m}, q^{-N}};q;
{\frac{c^m q }{a_1 a_2 B}}
\right]
\nonumber
\end{eqnarray}
and
\begin{eqnarray}\label{Phi-W-1}
\MP{1,m}{N}
{ a \\ {\cdot}}
{ \{ b_k y_k \}_m \\ \{c y_k \}_m }
&=&
\frac
{(a)_N}
{(q)_N}
\prod_{1 \le k \le m}
\frac
{(b_{k} y_k)_N}
{(c y_k)_N}.
\end{eqnarray}
\end{lem}
\begin{proof}
It is not hard to see that in the case when $n=1$, $\Phi^{1,m}_N$ is the coefficient of $u^N$ in
(1-dimensional) basic hypergeometric series:
\[
{}_{m+1} \phi_m \left[
	\begin{matrix}
		a, \{ b_k y_k \}_{m}\\
 			\{ c y_k \}_{m}
\end{matrix}
;q; u
\right]. 
\] 
\eqref{Phi-W} can be obtained by setting $\gamma_{2} = N - \gamma_1$
and elementary series manipulations.
\end{proof}
We mention that the homogeneous part of $A_{n+1}$ basic hypergeometric series $\Phi^{n+1, m}_N$ 
can be expressed in terms of $A_n$ very well-poised basic hypergeometric series. 
This fact has first appeared as Lemma 1.22 in S.C. Milne`s
paper \cite{MilneH-VWP} with a slightly different notation.
Readers interested in the elliptic case might also have a look at
Proposition 3.2 in \cite{KajiNou} for $A_n$ elliptic hypergeometric series.

\medskip 

By using Lemma 4.1. and some replacement of parameters,
we obtain the following transformations for bilinear sums of 
very well-poised basic hypergeometric series from the master formula \eqref{MF}.

\medskip

\noindent
{The $n_1 = m_1 = n_2 = m_2 = 2$ case of \eqref{MF}, i.e.,}
\begin{eqnarray}\label{1d-GBL}
&&
\sum_{K \in \mathbb N}
\frac{(b / t, c / t, d_1 / t, d_2 / t,
 \sigma q, \epsilon, \phi, q^{-N} )_K}
{(q,  1/ t,  q/ e, q/f, 
\sigma q/ \beta, \sigma q / \gamma, 
 \sigma q / \delta_1, \sigma q / \delta_2)_K}
q^K
\\
& \times &
{}_{10 } W_9 
\left[
{t q^{-K}};
{b}, 
{ c, d_1,  d_2}, 
{ e q^{-K}, f q^{-K}, q^{-K}}; q; 
{\frac{t^{3} q^{3}}{b c d_1 d_2 e^{} f }}
\right]
\nonumber
\\
& \times &
{}_{10} W_9
\left[
{\sigma q^{K}  }; 
{\beta}, {\gamma, \delta_1, \delta_2 }, 
{\epsilon  q^{K }, \phi q^K, q^{K-N}};q;
{
\frac{\sigma^{3} q^{N+ 3}}
{\beta \gamma \delta_1 \delta_2 \epsilon \phi}
}
\right]
\nonumber
\\
&=&
\phi^N
\frac{{(\sigma q / \delta_1 \phi, \sigma q / \delta_2 \phi,
 \epsilon, \sigma q / \gamma \phi,
 \sigma q, \sigma q / \beta \phi)_N}}
{(\sigma q / \delta_1, \sigma q / \delta_2,
 \epsilon  / \phi , (\sigma q / \gamma,
\sigma q/ \phi, \sigma q / \beta )_N}
\nonumber
\\
&\times &
\sum_{L \in \mathbb N}
\frac{{({ t q /c f}, t q / bf,  t q / d_1 f, t q / d_2 f,
q^{1-N} \epsilon / \phi, q^{-N } \phi / \sigma, 
\phi, q^{-N})_L}}
{(q, {e / f}, t q/ f, q/ f, 
q^{-N} \gamma \phi / \sigma, q^{-N } \beta \phi / \sigma, 
q^{-N} \delta_1 \phi / \sigma, 
q^{-N} \delta_2 \phi / \sigma  )_L}
q^L
\nonumber
\\
& \times &
{}_{10} W_9
\left[
{q^{-L} f / e};  {t q / c e},  
{ t q / b e,  t q / d_1 e, t q / d_2 e },
{ f q^{-L} / t, 
f q^{-L}, q^{-L}};q;
{\frac{b c  d_1 d_2 e f q^{-1}}{t^{3}}}
\right]
\nonumber
\\
& \times &
{}_{10} W_9
\left[
{q^{L-N}  \phi / \epsilon  };
 {\sigma q / \gamma \epsilon}, 
{\sigma q / \beta \epsilon, 
\sigma q / \delta_1 \epsilon,  \sigma q / \delta_2 \epsilon },
{q^{-N + L } \phi / \sigma, q^L \phi, q^{L-N}};q;
{
\frac
{\beta \gamma \delta_1 \delta_2 \epsilon \phi q^{-N-1}}
{\sigma^{3} }
}
\right],
\nonumber
\end{eqnarray}
while provided with the ``totally balancing condition'' \eqref{tbc},
becomes after the parameter substitution
\begin{eqnarray}\label{1d-bcGBL}
{t^{3} }{\sigma^{3} q^{N+ 4}}
&=&
{b c d_1 d_2 e f }{\beta \gamma \delta_1 \delta_2 \epsilon \phi}.
\end{eqnarray}
%

{The $n_1 = m_1 = n_2  = 2$ and $m_2 = 1$ case of  \eqref{MF}}:
\begin{eqnarray}\label{1d-m21BL}
&&
\sum_{K \in \mathbb N}
\frac{{(b/t, c/t, d_1 / t, d_2 / t,
\sigma q, \phi, q^{-N})_K}}
{(q, 1/t, q / e,  q/f, 
\sigma q/ \beta, \sigma q / \delta_1, \sigma q / \delta_2)_K}
q^K
\\
& \times &
{}_{10} W_9
\left[
{ t q^{-K}  };
{b},
{c, d_1, d_2 },
{ e q^{-K}, 
 f q^{-K}, q^{-K}};q;
{\frac{t^3 q^{3} } 
{ b c d_1 d_2 e f}}
\right]
\nonumber
\\
& \times &
{}_8 W_7
\left[ 
{\sigma q^{K} }; 
{\beta},
{ \delta_{1}, \delta_2  },
{ \phi q^{K }, q^{K-N}};q;
{\frac{ \sigma q^{N+2}}
{ \beta \delta_1 \delta_2 \phi }}
\right]
\nonumber
\\
&=& \phi^N
\frac{(\sigma q, \sigma q / \beta \phi,
\sigma q / \delta_1 \phi, \sigma q / \delta_2 \phi )_N}
{(\sigma q / \beta, \sigma q / \phi, 
\sigma q / \delta_1, \sigma q / \delta_2)_N}
\nonumber
\\
&\times &
\sum_{L \in \mathbb N}
\frac
{( t q / b f, t q / c f, t q / d_1 f, t q / d_2 f,
q^{-N} \phi / \sigma, \phi, q^{-N} )_L}
{(q, e/f, t q/f, q/f, 
q^{-N} \phi \beta / \sigma, q^{-N} \delta_1 \phi / \sigma , 
q^{-N} \delta_2 \phi / \sigma )_L}
q^L
\nonumber
\\
& \times &
{}_{10} W_9 
\left[
{f q^{-L} / e  }; {t q / c e},
{ t q / b e, t q / d_1 e, t q / d_2 e},
{  f q^{-L} / t, 
f  q^{-L}, q^{-L}};q;
{\frac
{b c d_1 d_2 e^{} f q^{-1}}
{t^{3 }}}
\right]
\nonumber
\end{eqnarray}
while provided with the ``totally balancing condition''
\begin{eqnarray}\label{1d-bcm21BL}
{t^{3} }{\sigma^{2} q^{N+ 3}}
&=&
{b c d_1 d_2 e^{} f }{\beta  \delta_1 \delta_2  \phi}.
\end{eqnarray}

Furthermore, the $n_1  = n_2  = 2$ and $m_1 = m_2 = 1$ case of
\eqref{MF} becomes
\goodbreak
\begin{eqnarray}\label{1d-m1m21BL}
&&
\sum_{K \in \mathbb N}
\frac
{(b/t, d_1 / t, d_2 / t, 
\sigma q, \phi, q^{-N})_K}
{(q, 1/t, q/e,
\sigma q/ \beta, \sigma q / \delta_1, \sigma q / \delta_2)_K}
\\
& \times &
{}_8 W_7 
\left[
{ t q^{-K}  };
b, 
{ d_1, d_2  }, 
{ e q^{-K},  q^{-K}};q;
{\frac{ t^{2} q^2 }
{ b d_1 d_2 e}}
\right]
\nonumber
\\
& \times &
{}_8 W_7
\left[
{ \sigma q^{K}  };
\beta,
{ \delta_1, \delta_2 }, 
{  \phi q^{ K },  q^{K-N}};q;
{\frac{ \sigma^2  q^{N+2}}
{ \beta \delta_1 \delta_2 \phi }}
\right]
\nonumber
\\
&=& \phi^N
\frac{( \sigma q, \sigma q / \beta \phi,
\sigma q / \delta_1 \phi, \sigma q / \delta_2 \phi )_N}
{( \sigma q / \beta, \sigma q / \phi, 
 \sigma q / \delta_1, \sigma q / \delta_2 )_N}
\nonumber
\\
&\times &
{}_6 \phi_5
\left[
\begin{matrix}
{t q / b e,  t q / d_1 e, t q / d_2 e},
{q^{-N}  \phi / \sigma, \phi, q^{-N}}
\\
{t q / e, q/ e},
{ q^{-N}  \beta \phi / \sigma, q^{-N} \delta_1 \phi / \sigma,  
 q ^{-N} \delta_2 \phi / \sigma }
\end{matrix}
;q;q
\right]
\nonumber
\end{eqnarray}
together with the ``totally balancing condition''
\begin{eqnarray}\label{1d-bcm1m21BL}
{t^{2} }{\sigma^{2} q^{N+ 2}}
&=&
{b  d_1 d_2 e}{\beta  \delta_1 \delta_2  \phi}.
\end{eqnarray}

Finally, we note that in the case when $n_1 = m_1 = n_2 = m_2 =1$,
\eqref{MF} reduces to the Whipple-Sears transformation formula \eqref{SearsT1}.

\medskip

We propose to consider bilinear transformations as natural generalizations
of multivariate hypergeometric transformations which extend Whipple-Sears transformations 
\eqref{SearsT1}. In this sense, the hypergeometric transformations and summations which
 we have obtained in Section 5 and 6 in \cite{Kaji1} (see also\cite{KajiBDT}),
such as multiple Whipple-Watson type transformations between ${}_8 W_7$ series 
and terminating balanced ${}_4 \phi_3$ series,
Dougall-Jackson summation for terminating balanced ${}_8 W_7$ series and Bailey type transformations
for terminating balanced ${}_{10} W_9$ series, can be interpreted as natural generalizations
of $q$-Pfaff-Saalsch\"{u}tz summation formula for terminating balanced ${}_3 \phi_2$ series.
Besides those bilinear transformations,
the master formula \eqref{MF} involves various types 
of transformation formulas for bilinear and linear sums of multivariate basic hypergeometric series.
Furthermore, we can obtain another master formula of alternate type 
and $q$-Pfaff-Saalsch\"{u}tz summation type and can deduce a family of transformations for 
bilinear and linear sums of very well-poised hypergeometric series from them
in the similar manner as in this paper.  
Details and other transformation formulas including linear sums and multivariate generalizations
will be given elsewhere.

}
\vspace{10mm}

\noindent
{\large  \bf Acknowledgments}

\medskip

The author expresses his sincere thanks to Professors Mourad Ismail, Erik Koelink and 
Masatoshi Noumi for comments. 


\vspace{15mm}

Department of Mathematics, Kobe University, Rokko-dai, Kobe
657-8501, Japan


\end{document}